\documentclass{article}

\usepackage{amsfonts,amsmath, amsthm, amssymb}
\usepackage[latin1]{inputenc}
\usepackage[english]{babel}
\usepackage{graphicx}
\usepackage{enumerate}

\newcommand{\D}{\mathcal D}

\newcommand{\C}{\mathcal C}

\newcommand{\Q}{\mathcal Q}
\newcommand{\R}{\mathcal R}

\newcommand{\sub}{\subseteq}

\newcommand{\comment}[1]{}

\newtheorem{theorem}{Theorem}
\newtheorem{lemma}[theorem]{Lemma}
\newtheorem{corollary}[theorem]{Corollary}
\newtheorem{proposition}[theorem]{Proposition}

\theoremstyle{definition}

\def\?#1{\vadjust{\vbox to 0pt{\vss\vskip-8pt\leftline{%
     \llap{\hbox{\vbox{\pretolerance=-1
     \doublehyphendemerits=0\finalhyphendemerits=0
     \hsize16truemm\tolerance=10000\small
     \lineskip=0pt\lineskiplimit=0pt
     \rightskip=0pt plus16truemm\baselineskip8pt\noindent
     \hskip0pt        
     #1\endgraf}\hskip7truemm}}}\vss}}}
     
     \title{Algebraically grid-like graphs have large tree-width}
     \author{Daniel Wei\ss{}auer}
     \date{Department of Mathematics\\ University of Hamburg}

\begin{document}

	\maketitle     
     
	\begin{abstract}
	
	By the Grid Minor Theorem of Robertson and Seymour, every graph of sufficiently large tree-width contains a large grid as a minor. Tree-width may therefore be regarded as a measure of 'grid-likeness' of a graph.
 	
 	The grid contains a long cycle on the perimeter, which is the $\mathbb{F}_2$-sum of the rectangles inside. Moreover, the grid distorts the metric of the cycle only by a factor of two. We prove that every graph that resembles the grid in this algebraic sense has large tree-width:
 	
 	Let $k, p$ be integers, $\gamma$ a real number and~$G$ a graph. Suppose that~$G$ contains a cycle of length at least~$2 \gamma p k$ which is the $\mathbb{F}_2$-sum of cycles of length at most~$p$ and whose metric is distorted by a factor of at most~$\gamma$. Then~$G$ has tree-width at least~$k$.
 	
	\end{abstract}	     
     
     \begin{section}{Introduction}


     For a positive integer~$n$, the $(n \times n)$-grid is the graph~$G_n$ whose vertices are all pairs $(i,j)$ with $1 \leq i, j  \leq n$, where two points are adjacent when they are at Euclidean distance~1. The cycle~$C_n$, which bounds the outer face in the natural drawing of~$G_n$ in the plane, has length $4(n-1)$ and is the $\mathbb{F}_2$-sum of the rectangles bounding the inner faces. This is by itself not a distinctive feature of graphs with large tree-width: The situation is similar for the $n$-wheel~$W_n$, the graph consisting of a cycle~$D_n$ of length~$n$ and a vertex $x \notin D_n$ which is adjacent to every vertex of~$D_n$. There, $D_n$ is the $\mathbb{F}_2$-sum of all triangles $xyz$ for $yz \in E(D_n)$. Still, $W_n$ only has tree-width~3.
     
     
     The key difference is the fact that in the wheel, the metric of the cycle is heavily distorted: any two vertices of~$D_n$ are at distance at most two within~$W_n$, even if they are far apart within~$D_n$. In the grid, however, the distance between two vertices of~$C_n$ within~$G_n$ is at least half of their distance within~$C_n$.
     
     In order to incorporate this factor of two and to allow for more flexibility, we equip the edges of our graphs with lengths. For a graph~$G$, a \emph{length-function on~$G$} is simply a map $\ell : E(G) \to \mathbb{R}_{>0}$. We then define the \emph{$\ell$-length} $\ell(H)$ of a subgraph $H \sub G$ as the sum of the lengths of all edges of~$H$. This naturally induces a notion of distance between two vertices of~$G$, where we define~$d_G^{\ell}$ as the minimum $\ell$-length of a path containing both. A subgraph $H \sub G$ is \emph{$\ell$-geodesic} if it contains a path of length $d_G^{\ell}(a,b)$ between any two vertices $a, b \in V(H)$. 
          
          When no length-function is specified, the notions of length, distance and geodecity are to be read with respect to $\ell \equiv 1$ constant.
          
          On the grid-graph~$G_n$, consider the length-function~$\ell$ which is equal to~1 on $E(C_n)$ and assumes the value~2 elsewhere. Then~$C_n$ is $\ell$-geodesic of length $\ell(C_n) = 4(n-1)$ and the sum of cycles of $\ell$-length at most~8. We show that any graph which shares this algebraic feature has large tree-width.
     
	     	\begin{theorem} \label{main theorem lengths}
	     		Let~$k$ be a positive integer and $r > 0$. Let~$G$ be a graph with rational-valued length-function~$\ell$. Suppose~$G$ contains an $\ell$-geodesic cycle~$C$ with $\ell(C) \geq 2 r k$, which is the $\mathbb{F}_2$-sum of cycles of $\ell$-length at most~$r$. Then the tree-width of~$G$ is at least~$k$.
	     	\end{theorem}
	     
     	     	 	The starting point of Theorem~\ref{main theorem lengths} was a similar result of Matthias Hamann and the author~\cite{HW16sidma}. There, it is assumed that not only the fixed cycle~$C$, but the whole cycle space of~$G$ is generated by short cycles. 
	
	\begin{theorem}[{{\cite[Corollary~3]{HW16sidma}}}] \label{bounding ctw}
	Let~$k, p$ be positive integers. Let~$G$ be a graph whose cycle space is generated by cycles of length at most~$p$. If~$G$ contains a geodesic cycle of length at least~$kp$, then the tree-width of~$G$ is at least~$k$.
\end{theorem}

     	It should be noted that Theorem~\ref{bounding ctw} is not implied by Theorem~\ref{main theorem lengths}, as the constant factors are different. In fact, the proofs are also quite different, although Lemma~\ref{using cycle gen} below was inspired by a similar parity-argument in~\cite{HW16sidma}.
     	
     	It is tempting to think that, conversely, Theorem~\ref{main theorem lengths} could be deduced from Theorem~\ref{bounding ctw} by adequate manipulation of the graph~$G$, but we have not been successful with such attempts.
    
     	\end{section}

     \begin{section}{Proof of Theorem~\ref{main theorem lengths}}
     
         	The relation to tree-width is established via a well-known separation property of graphs of bounded tree-width, due to Robertson and Seymour~\cite{GMII}.
     	
     	\begin{lemma}[\cite{GMII}] \label{tw separator}
		 Let~$k$ be a positive integer, $G$ a graph and $A \sub V(G)$. If the tree-width of~$G$ is less than~$k$, then there exists $X \sub V(G)$ with $|X| \leq k$ such that every component of $G - X$ contains at most $|A \setminus X|/2$ vertices of~$A$.
		\end{lemma}
     
		It is not hard to see that Theorem~\ref{main theorem lengths} can be reduced to the case where $\ell \equiv 1$. This case is treated in the next theorem.

          \begin{theorem} \label{precise theorem}
     	Let $k, p$ be positive integers. Let~$G$ be a graph containing a geodesic cycle~$C$ of length at least $4\lfloor p/ 2 \rfloor k$, which is the $\mathbb{F}_2$-sum of cycles of length at most~$p$. Then for every $X \sub V(G)$ of order at most~$k$, some component of $G - X$ contains at least half the vertices of~$C$.
     \end{theorem}
     

	\begin{proof}[Proof of Theorem~\ref{main theorem lengths}, assuming Theorem~\ref{precise theorem}]
	
	Let~$\D$ be a set of cycles of length at most~$r$ with $C = \bigoplus \D$. 
	
%
	
	Since~$\ell$ is rational-valued, we may assume that $r \in \mathbb{Q}$, as the premise also holds for~$r'$ the maximum $\ell$-length of a cycle in~$\D$. Take an integer~$M$ so that $rM$ and $\ell' (e) := M \ell(e)$ are natural numbers for every $e\in E(G)$.
	
	Obtain the subdivision~$G'$ of~$G$ by replacing every $e \in E(G)$ by a path of length~$\ell'(e)$. Denote by~$C', D'$ the subdivisions of~$C$ and $D \in \D$, respectively. Then $C' = \bigoplus_{D \in \D} D'$ and $|C'| = M \ell(C) \geq 2 (Mr) k$, while $|D'| = M \ell(D) \leq Mr$ for every $D \in \D$. By Theorem~\ref{precise theorem}, for every $X \sub V(G')$ with $|X | \leq k$ there exists a component of $G' - X$ that contains at least half the vertices of~$C'$. By Lemma~\ref{tw separator}, $G'$ has tree-width at least~$k$. Since tree-width is invariant under subdivision, the tree-width of~$G$ is also at least~$k$.
	\end{proof}

     Our goal is now to prove Theorem~\ref{precise theorem}. The proof consists of two separate lemmas. The first lemma involves separators and $\mathbb{F}_2$-sums of cycles.

	  \begin{lemma} \label{using cycle gen}
     	Let~$G$ be a graph, $C \sub G$ a cycle and~$\D$ a set of cycles in~$G$ such that $C = \bigoplus \D$. Let~$\R$ be a set of disjoint vertex-sets of~$G$ such that for every $R \in \R$, $R \cap V(C)$ is either empty or induces a connected subgraph of~$C$. Then either some $D \in \D$ meets two distinct $R, R' \in \R$ or there is a component~$Q$ of $G - \bigcup \R$ with $V(C) \sub V(Q) \cup \bigcup \R$.
     \end{lemma}		
     
     \begin{proof}
		Suppose that no $D \in \D$ meets two distinct $R, R' \in \R$. Then~$C$ has no edges between the sets in~$\R$: Any such edge would have to lie in at least one $D \in \D$. 	Let $Y := \bigcup \R$ and let~$\Q$ be the set of components of $G - Y$. 
		
		Let $Q \in \Q$, $R \in \R$ and $D \in \D$ arbitrary. If~$D$ has an edge between~$Q$ and~$R$, then~$D$ cannot meet $Y \setminus R$. Therefore, all edges of~$D$ between~$Q$ and $V(G) \setminus Q$ must join~$Q$ to~$R$. As~$D$ is a cycle, it has an even number of edges between~$Q$ and $V(G) \setminus Q$ and thus between~$Q$ and~$R$. As $C = \bigoplus \D$, we find
		\[
		e_C(Q,R) \equiv \sum_{D \in \C} e_D(Q,R) \equiv 0 \, \mod 2.
		\]
		For every $R \in \R$ which intersects~$C$, there are precisely two edges of~$C$ between~$R$ and $V(C) \setminus R$, because $R \cap C$ is connected. As mentioned above, $C$ contains no edges between~$R$ and $Y \setminus R$, so both edges join~$R$ to $V(G) \setminus Y$. But~$C$ has an even number of edges between~$R$ and each component of $V(G) \setminus Y$, so it follows that both edges join~$R$ to the same $Q(R) \in \Q$.
		
		Since every component of $C - (C \cap Y)$ is contained in a component of $G - Y$, it follows that there is a $Q \in \Q$ containing all vertices of~$C$ not contained in~$Y$.
		\end{proof}
     
     
   To deduce Theorem~\ref{precise theorem}, we want to apply Lemma~\ref{using cycle gen} to a suitable family~$\R$ with $\bigcup \R \supseteq X$ to deduce that some component of $G - X$ contains many vertices of~$C$. Here, $\D$ consists of cycles of length at most~$\ell$, so if the sets in~$\R$ are at pairwise distance~$>\! \lfloor \ell / 2 \rfloor$, then no $D \in \D$ can pass through two of them. The next lemma ensures that we can find such a family~$\R$ with a bound on $|\bigcup \R|$, when the cycle~$C$ is geodesic.

     			\begin{lemma} \label{extend separator}
		Let~$d$ be a positive integer, $G$ a graph, $X \sub V(G)$ and $C \sub G$ a geodesic cycle. Then there exists a family~$\R$ of disjoint sets of vertices of~$G$ with $X \sub \bigcup \R \sub X \cup V(C)$ and $| \bigcup \R \cap V(C) | \leq 2 d |X|$ such that for each $R \in \R$, the set $R \cap V(C)$ induces a (possibly empty) connected subgraph of~$C$ and the distance between any two sets in~$\R$ is greater than~$d$. 
	\end{lemma}

	\begin{proof}
	Let $Y \sub V(G)$ and $y \in Y$. For $j \geq 0$, let~$B_Y^j(y)$ be the set of all $z \in Y$ at distance at most~$j d$ from~$y$. Since $|B_Y^0(y)| = 1$, there is a maximum number~$j$ for which $|B_Y^j(y)| \geq 1 + j$, and we call this $j = j_Y(y)$ the \emph{range of~$y$ in~$Y$}. Observe that every $z \in Y \setminus B^{j_Y(y)}$ has distance greater than $(j_Y(y)+1) d$ from~$y$.
	
	Starting with $X_1 := X$, repeat the following procedure for $k \geq 1$. If $X_k \cap V(C)$ is empty, terminate the process. Otherwise, pick an $x_k \in X_k \cap V(C)$ of maximum range in~$X_k$. Let $j_k := j_{X_k}(x_k)$ and $B_k := B_{X_k}^{j_k}(x_k)$. Let $X_{k+1} := X_k \setminus B_k$ and repeat.
	
	Since the size of~$X_k$ decreases in each step, there is a smallest integer~$m$ for which $X_{m+1} \cap V(C)$ is empty, at which point the process terminates. By construction, the distance between~$B_k$ and~$X_{k+1}$ is greater than~$d$ for each $k \leq m$. For each $1 \leq k \leq m$, there are two edge-disjoint paths $P_k^1, P_k^2 \sub C$, starting at~$x_k$, each of length at most~$j_k d$, so that $B_k \cap V(C) \sub S_k := P_k^1 \cup P_k^2$. Choose these paths minimal, so that the endvertices of~$S_k$ lie in~$B_k$. Note that every vertex of~$S_k$ has distance at most $j_k d$ from~$x_k$. Therefore, the distance between $R_k := B_k \cup S_k$ and~$X_{k+1}$ is greater than~$d$.
	
	We claim that the distance between~$R_k$ and~$R_{k'}$ is greater than~$d$ for any $k < k'$. Since $B_{k'} \sub X_{k+1}$, it is clear that every vertex of~$B_{k'}$ has distance greater than~$d$ from~$R_k$. Take a vertex $q \in S_{k'} \setminus R_{k'}$ and assume for a contradiction that its distance to~$R_k$ was at most~$d$. Then the distance between~$x_k$ and~$q$ is at most $(j_k + 1)d$. Let $a, b \in B_{k'}$ be the endvertices of~$S_{k'}$. If $x_k \notin S_{k'}$, then one of~$a$ and~$b$ lies on the shortest path from~$x_k$ to~$q$ within~$C$ and therefore has distance at most~$(j_k+1) d$ from~$x_k$. But then, since~$j_k$ is the range of~$x_k$ in~$X_k$, that vertex would already lie in~$B_k$, a contradiction. Suppose now that $x_k \in S_{k'}$. Then~$x_k$ lies on the path in~$S_{k'}$ from~$x_k$ to one of~$a$ or~$b$, so the distance between~$x_k$ and~$x_{k'}$ is at most $j_{k'} d$. Since $x_{k'} \in X_k \cap V(C)$, it follows from our choice of~$x_k$ that 
	\[
	j_k = j_{X_k}(x_k) \geq j_{X_k}(x_{k'}) \geq j_{X_{k'}}(x_{k'}) = j_{k'} ,
	\]
	where the second inequality follows from the fact that $X_{k'} \sub X_k$ and $j_Y(y) \geq j_{Y'}(y)$ whenever $Y \supseteq Y'$. But then $x_{k'} \in B_k$, a contradiction. This finishes the proof of the claim.
	
	Finally, let $\R := \{ R_k \colon 1 \leq k \leq m \} \cup \{ X_{m+1} \}$. The distance between any two sets in~$\R$ is greater than~$d$. For $k \leq m$, $R_k \cap V(C) = S_k$ is a connected subgraph of~$C$, while $X_{m+1} \cap V(C)$ is empty. Moreover,
			\begin{align*}
			| \bigcup \R \cap V(C)| &= \sum_{k=1}^m |S_k| \leq \sum_{k=1}^m  (1 + 2 j_k d) \\
								&\leq \sum_{k=1}^m (1 + 2(|B_k|-1) d ) \\
								&\leq \sum_{k=1}^m 2 |B_k| d \leq 2 d |X|  .
		\end{align*}
	\end{proof}

		\begin{proof}[Proof of Theorem~\ref{precise theorem}]
		Let $X \sub V(G)$ of order at most~$k$ and let $d := \lfloor p /2 \rfloor$. By Lemma~\ref{extend separator}, there exists a family~$\R$ of disjoint sets of vertices of~$G$ with $X \sub \bigcup \R \sub X \cup V(C)$ and $| \bigcup \R \cap V(C)| \leq 2 d k$ so that for each $R \in \R$, the set $R \cap V(C)$ induces a (possibly empty) connected subgraph of~$C$ and the distance between any two sets in~$\R$ is greater than~$d$.
		
		Let~$\D$ be a set of cycles of length at most~$p$ with $C = \bigoplus \D$. Then no $D \in \D$ can meet two distinct $R, R' \in \R$, since the diameter of~$D$ is at most~$d$. By Lemma~\ref{using cycle gen}, there is a component~$Q$ of $G - \bigcup \R$ which contains every vertex of $C \setminus \bigcup \R$. This component is connected in $G - X$ and therefore contained in some component~$Q'$ of $G - X$, which then satisfies
		\[
		|Q' \cap V(C)| \geq  |C| - |\bigcup \R \cap V(C)| \geq |C| - 2 d k .
		\]
		Since $|C| \geq 4 d k$, the claim follows.
\end{proof}

     \end{section}

     \begin{section}{Remarks}
     
     We have described the content of Theorem~\ref{main theorem lengths} as an \emph{algebraic} criterion for a graph to have large tree-width. The reader might object that the cycle~$C$ being $\ell$-geodesic is a metric property and not an algebraic one. Karl Heuer has pointed out to us, however, that geodecity of a cycle can be expressed as an algebraic property after all. This is a consequence of a more general lemma of Gollin and Heuer~\cite{geodesiccuts}, which allowed them to introduce a meaningful notion of geodecity for cuts.
     
     \begin{proposition}[\cite{geodesiccuts}]
     	Let~$G$ be a graph with length-function~$\ell$ and $C \sub G$ a cycle. Then~$C$ is $\ell$-geodesic if and only if there do not exist cycles $D_1, D_2$ with $\ell(D_1), \ell(D_2) < \ell(C)$ such that $C = D_1 \oplus D_2$.
     \end{proposition}

	     	Finally, we'd like to point out that Theorem~\ref{main theorem lengths} does not only offer a 'one-way criterion' for large tree-width, but that it has a qualitative converse. First, we recall the Grid Minor Theorem of Robertson and Seymour~\cite{GMV}, phrased in terms of walls . For a positive integer~$t$, an \emph{elementary $t$-wall} is the graph obtained from the $2t \times t$-grid as follows. Delete all edges with endpoints $(i,j), (i,j+1)$ when~$i$ and~$j$ have the same parity. Delete the two resulting vertices of degree one. A \emph{$t$-wall} is any subdivision of an elementary $t$-wall. Note that the $(2t \times 2t)$-grid has a subgraph isomorphic to a $t$-wall.
	     	
	     	\begin{theorem}[Grid Minor Theorem~\cite{GMV}]
	     		For every~$t$ there exists a~$k$ such that every graph of tree-width at least~$k$ contains a $t$-wall.
	     	\end{theorem}
	     	
	     	Here, then, is our qualitative converse to Theorem~\ref{main theorem lengths}, showing that the algebraic condition in the premise of Theorem~\ref{main theorem lengths} in fact captures tree-width. 	
	     	
	     	\begin{corollary}
	     	For every~$L$ there exists a~$k$ such that for every graph~$G$ the following holds. If~$G$ has tree-width at least~$k$, then there exists a rational length-function on~$G$ so that~$G$ contains a $\ell$-geodesic cycle~$C$ with $\ell(C) \geq L$ which is the $\mathbb{F}_2$-sum of cycles of $\ell$-length at most~$1$.
	     	\end{corollary}
	     	
	     	\begin{proof}
	     		Let $s := 3L$. By the Grid Minor Theorem, there exists an integer~$k$ so that every graph of tree-width at least~$k$ contains an $s$-wall. Suppose ~$G$ is a graph of tree-width at least~$k$. Let~$W$ be an elementary $s$-wall so that~$G$ contains some subdivision~$W'$ of~$W$, where $e \in E(W)$ has been replaced by some path $P^e \sub G$ of length~$m(e)$.

	     		The outer cycle~$C$ of~$W$ satisfies $d_C(u,v) \leq 3 d_W(u,v)$ for all $u,v \in V(C)$. Moreover, $C$ is the $\mathbb{F}_2$-sum of cycles of length at most~six.
	     		
	     		Define a length-function~$\ell$ on~$G$ as follows. Let $e \in E(G)$. If $e \in P^f$ for $f \in E(C)$, let $\ell(e) := 1/m(f)$. Then $\ell(P^f) = 1$ for every $f \in E(C)$. If $e \in P^f$ for $f \in E(W) \setminus E(C)$, let $\ell(e) := 3/m(f)$. Then $\ell(P^f) = 3$ for every $f \in E(W) \setminus E(C)$. If $e \notin E(W')$, let $\ell(e) := 10s^3$, so that $\ell(e) > \ell(W')$.
	     		
	     		It is easy to see that the subdivision $C' \sub G$ of~$C$ is $\ell$-geodesic in~$G$. It has length $\ell(C') = |C| \geq 6s$ and is the $\mathbb{F}_2$-sum of the subdivisions of 6-cycles of~$W$. Each of these satisfies $\ell(D) \leq 18$. Rescaling all lengths by a factor of~$1/18$ yields the desired result.
	     		
	     	\end{proof}

     \end{section}

\bibliographystyle{plain}
\bibliography{collective}

     \end{document}